\documentclass[12pt]{amsart}

\usepackage{bm}
\usepackage{fullpage}
\usepackage{amssymb}
\usepackage{mathrsfs}
\usepackage{amsmath,amsfonts,amsthm}
\usepackage{color}

\newtheorem{alphtheorem}{Theorem}

\newtheorem{lemma}{Lemma}

\newtheorem{theorem}{Theorem}
\newtheorem{corollary}{Corollary}

\def\I{\mathcal{I}}

\def\P{\mathbb{P}}

\def\S{\mathcal{S}}

\def\HH{\mathcal{H}}

\def\A{\mathcal{A}}

\def\C{\mathcal{C}}

\def\P{\mathbb{P}}

\def\KG{\operatorname{KG}}

\title{On the random version of the Erd{\H o}s matching conjecture}
\author{Meysam Alishahi}
\address{M. Alishahi, 
Faculty of Mathematical Sciences,
Shahrood University of Technology, Shahrood, Iran \&
School of Mathematics, Institute for Research in Fundamental Sciences (IPM), P.O. Box 19395-5746, Tehran, Iran}
\email{meysam\_alishahi@shahroodut.ac.ir}

\author{Ali Taherkhani}
\address{A. Taherkhani, 
Department of Mathematics, Institute for Advanced Studies in Basic Sciences (IASBS), Zanjan 45137-66731, Iran}
\email{ali.taherkhani@iasbs.ac.ir}

\begin{document}
\maketitle 

\begin{abstract}
The Kneser hypergraph $\KG^r_{n,k}$ is an $r$-uniform hypergraph with vertex set consisting of all $k$-subsets of $\{1,\ldots,n\}$ and 
any collection of $r$ vertices forms an edge if their corresponding $k$-sets are pairwise disjoint. 
The random Kneser hypergraph $\KG^r_{n,k}(p)$ is a spanning subhypergraph of $\KG^r_{n,k}$ in which each edge of $\KG^r_{n,k}$ is retained independently of each other 
with probability $p$. The independence number of random subgraphs of $\KG^2_{n,k}$ was recently addressed in a series of works  by Bollob{\'a}s, Narayanan, and Raigorodskii~(2016), 
Balogh, Bollob{\'a}s, and Narayanan~(2015), Das and Tran~(2016), and Devlin and Kahn~(2016). It was proved that the random counterpart of the Erd{\H o}s-Ko-Rado theorem 
continues to be valid even for very small values of $p$. In this paper, generalizing this result, we will investigate the independence number of  random Kneser hypergraphs $\KG^r_{n,k}(p)$. Broadly speaking, when $k$ is much smaller that $n$, we will prove that the random analogue of the Erd{\H o}s matching conjecture is true 
even for extremely small values of $p$. \\ 
\end{abstract}

\section{Motivations and Main Results}
Let $n,k$ and $r$ be three positive integers such that $n\geq 2k$ and $r\geq 2$. Throughout the paper, the two symbols $[n]$ and ${[n]\choose k}$ respectively stand for the sets 
$\left\{1,\ldots,n\right\}$ and $\left\{A\subseteq[n]\colon |A|=k\right\}$.  The  {\it Kneser hypergraph $\KG^{r}_{n,k}$} 
is an $r$-uniform hypergraph whose vertex set is ${[n]\choose k}$ and its edge set consists of all 
pairwise  disjoint $r$-tuples  of  elements in ${[n]\choose k}$, i.e., 
$$E(\KG^r_{n,k})=\Big\{\{A_1,\ldots,A_r\}\colon A_1,\ldots,A_r\in {[n]\choose k} \text{ are pairwise  disjoint}\Big\}.$$ 
For each $x\in[n]$, the set $\S_x=\left\{A\in {[n]\choose k}\colon x\in A\right\}$ is called a {\it star}. It is clear that any star is an independent set of 
$\KG^{2}_{n,k}$, that is, a set of vertices containing no edge. 
 We remind the reader that the maximum size of an independent set in a hypergraph $\HH$ is called the {\it   independence number of} $\HH$, denoted by $\alpha(\HH)$. 
The seminal Erd{\H o}s-Ko-Rado theorem states that  for $n\geq 2k$, the independence number  of $\KG^2_{n,k}$ is ${n-1 \choose k-1}$;
 furthermore   if  $n>2k$, the only independent sets of this size are  the stars. 
As an extension of the Erd{\H o}s-Ko-Rado theorem, Erd{\H o}s~\cite{Erdos65} conjectured  
that $\alpha(\KG^r_{n,k})=\max \left\{{rk-1\choose k},{n\choose k}-{n-r+1\choose k}\right\}$ provided that $n\geq rk-1$.  
Easy computation shows that for $n\geq r(k+{1\over 2})$, the aforementioned maximum is ${n\choose k}-{n-r+1\choose k}$. 
In recent years, this conjecture  has received significant attention
and several papers were devoted to the study of this conjecture; see, e.g.,~\cite{BolDayErd76,Erdos65, ErdGal59,Frankl13,Frankl2017, FranLucMie12,HuaLohSud12,LucMie14}. 
Regarding this conjecture, the best known result is proved by 
Frankl~\cite{Frankl13}. Provided $n\geq (2r-1)k-r+1$, he proved that $\alpha(\KG^r_{n,k})={n\choose k}-{n-r+1\choose k}$; furthermore, 
any independent set of this size is formed by the union of some $r-1$ distinct stars which confirms the conjecture in this range. 
For more recent results concerning this conjecture, one can refer to~\cite{FK16,FK17}. 
It is worth noting  that there is another interesting extension of the Erd{\H o}s-Ko-Rado theorem due to Hilton and Milner~\cite{HM} asserting that for $n>2k$, any independent set of Kneser graph $KG_{n,k}$ which is contained in no star has cardinality at most
${n-1\choose k-1}-{n-k-1\choose k-1}+1$.  For recent results, one can see~\cite{F18, KZ18}.

Let $\KG^{r}_{n,k}(p)$ be the  random subhypergraph of $\KG^{r}_{n,k}$  whose 
vertex set is the same as $\KG^{r}_{n,k}$  and each edge of $\KG^{r}_{n,k}$ is retained independently of each other with probability $p$. 
Throughout the paper, when $r=2$, we shall drop the super-index $r$ and write $\KG_{n,k}$ and $\KG_{n,k}(p)$ instead of $\KG^r_{n,k}$ and $\KG^r_{n,k}(p)$, respectively. 
Also, we say an event occurs {\it with high probability} or {\it likely} happens if  
it can be made as close as desired to $1$ by making $n$ large enough. 

As a fast growing branch of hypergraph theory, many articles are recently devoted to investigating the properties of random Kneser hypergraphs $\KG^r_{n,k}(p)$; see~\cite{AH2018,MR3482268,BGPR14, BGPR15,MR3403515,MR3504983,MR3513857,KR17,K16, K17, P15, P17, PR16,R17}. 
Extending some results in~\cite{BGPR14,BGPR15}, Bollob{\'a}s, Narayanan and Raigorodskii \cite{MR3403515}  studied the independence number of random Kneser graphs $\KG_{n,k}(p)$.
They tried to answer the question that 
for which $p$, the Erd{\H o}s-Ko-Rado theorem is likely valid in $\KG_{n,k}(p)$. Surprisingly,  when $k$ is much smaller than $n$, they proved that 
an analogue of the Erd{\H o}s-Ko-Rado theorem continues to hold even after deleting practically all the edges of the Kneser graphs.
\begin{alphtheorem}\cite[Theorem~1.2]{MR3403515}\label{bela1}
Fix a real number $\varepsilon>0$ and let $k=k(n)$ be a natural number such that $2\leq k=o(n^{1/3})$. Then as $n\to \infty$, 
$$
\mathbb{P}\left(\alpha\left(\KG_{n,k}(p)\right)={n-1\choose k-1}\right)\longrightarrow
\left\{\begin{array}{lll}
               1        &\quad       & p\geq (1+\varepsilon){(r+1)\ln n-r\ln r\over {n-1\choose k-1}}\\ \\
               0        &\quad      & p\leq (1-\varepsilon){(r+1)\ln n-r\ln r\over {n-1\choose k-1}}.
\end{array}\right.
$$ 

Furthermore, when $p\geq (1+\varepsilon){(r+1)\ln n-r\ln r\over {n-1\choose k-1}}$, 
with high probability, the only independent sets of size  ${n-1\choose k-1}$  in $\KG_{n,k}(p)$ are the stars.
\end{alphtheorem}
In addition, they conjectured that a similar result should hold for $k=o(n)$ which first was partially answered by Balogh, Bollob{\'a}s and Narayanan~\cite{MR3482268}.
Then, a significantly sharper result was proved by Das  and Tran~\cite{MR3504983}. They extended  the Bollob{\'a}s-Narayanan-Raigorodskii theorem to $k$ as large as linear in $n$ 
subsuming the earlier results. Finally,  Delvin and Kahn~\cite{MR3513857} extended this theorem to general $k$ with $n\geq 2k+2$.  Also, for $n=2k+1$, they proved that 
there is a fixed $p<1$ such that, with high probability, $\alpha(\KG_{2k+1,k}(p))={2k\choose k-1}$ and the stars are the only maximum independent sets. 
It is worth mentioning that some other kinds of generalizations of Theorem~\ref{bela1} can be found in~\cite{P15, P17, PR16, R17}. 

Seeing the Erd{\H o}s matching conjecture as a generalization of the Erd{\H o}s-Ko-Rado theorem to the case of Kneser hypergraphs, 
one may naturally ask for which $p>0$ the Erd{\H o}s matching conjecture continues likely to hold in $\KG^r_{n,k}(p)$. 
Mainly motivated by this question, in this paper, 
we shall investigate the  size and structure of maximum independent sets in  random Kneser hypergraphs. We will show that 
the random counterpart of the Erd{\H o}s matching conjecture continues to hold when $k$ is very small in comparison to $n$.
More precisely, when $r\geq 2$, we shall prove a hypergraph version of Theorem ~\ref{bela1} which in part implies a slightly weaker version of this theorem.
It should be mentioned that our technique in the proof of this result is different from that of Theorem~\ref{bela1} in~\cite{MR3403515}. 
A natural candidate for the probability threshold could be obtained by seeking for a threshold $p_c$ such that 
for each positive constant $\varepsilon$, if $p\leq (1-\varepsilon)p_c$, then the expected number of 
independent sets $\A$ in $\KG^r_{n,k}(p)$ of size ${n\choose k}-{n\choose n-r+1}+1$ which contain some $r-1$ distinct stars
 goes to zero as $n$ tends to infinity. 
Since for any such family $\A$, we have $|E(\KG^r_{n,k}[\A])|=\prod\limits_{i=1}^{r-1}{n-ik-(r-i)\choose k-1}$, 
the expected number of such independent sets would be 
$${n\choose r-1}{n-r+1\choose k}(1-p)^{\prod\limits_{i=1}^{r-1}{n-ik-(r-i)\choose k-1}}$$
which clearly suggests  $p_c={\ln\left({n\choose r-1}{n-r+1\choose k}\right)\over \prod\limits_{i=1}^{r-1}{n-ik-(r-i)\choose k-1}}$. 
Our main result is the following theorem. 
\begin{theorem}\label{maintheorem}
Let $n, k$ and $r$ be positive integers such that $k=k(n)\geq 2$, $r\geq 2$, and $n\geq r(k+{1\over 2})$. 
\begin{description}
\item[I] There are positive constants $\zeta=\zeta(r)$ and $C=C(r)$ such that 
$$\mathbb{P}\left(\alpha\left(\KG^r_{n,k}(p)\right)={n\choose k}-{n-r+1\choose k}\right)\rightarrow 1$$ 
provided $p> \zeta p_c$ and $k\leq C n^{1\over 7}$ {\rm(}$k=o(n^{1\over 3})$ for $r=2,3${\rm)}. 

Furthermore, with high probability, the only independent sets of size  ${n\choose k}-{n-r+1\choose k}$ are the trivial ones, namely the union of $r-1 $ distinct stars.   \\

\item[II] For each $\varepsilon \in(0,1]$, we have 
$$\mathbb{P}\left(\alpha\left(\KG^r_{n,k}(p)\right)={n\choose k}-{n-r+1\choose k}\right)\rightarrow0$$
provided $p\leq  (1-\varepsilon){p_c}$.
\end{description}
\end{theorem}
This theorem generalizes Theorem~\ref{bela1} to the case of Kneser hypergraphs. 
As stated above (see the discussion after Theorem~\ref{bela1}), owing to the works~\cite{MR3482268,MR3504983,MR3513857}, 
Theorem~\ref{bela1} has been extended to $k$ as large as  ${n\over 2}-2$. 
We believe that the condition on $k$ in Theorem~\ref{maintheorem} is superfluous as well. 
By the way, we conjecture that the same formula for the critical threshold continues to work for $r\geq 3$ and $n> r(k+{1\over 2})$, but we are unable to prove this presently. 
Also, for $1\leq n-rk \leq {r\over 2}$, it is interesting to study the behavior of $\alpha(\KG^r_{n,k}(p))$. Note that the case $r=2$ is already addressed by the aforementioned result by Delvin and Kahn~\cite{MR3513857}. Indeed, for $1\leq n-rk \leq {r\over 2}$, we surmise that there is a constant $p<1$ such that, with high probability,    
$\alpha(\KG^r_{n,k}(p))$ is equal to  $\max \left\{{rk-1\choose k},{n\choose k}-{n-r+1\choose k}\right\}$ and the only maximum independent sets are the trivial ones.

\bigskip
The rest of the paper is organized as follows. Section~\ref{proofsec} is devoted to the proof of Theorem~\ref{maintheorem} which is divided into three subsections.
In the first subsection, we set up some notations, then the proof of the first and the second parts of the theorem will be discussed separately.
\section{Proof of Theorem~\ref{maintheorem}}\label{proofsec}
\subsection{Notation}
For two functions $f(n)$ and $g(n)$, we write $f\sim g$ and $f=o(g)$ whenever $\lim\limits_{n\rightarrow \infty}{f\over g}=1$ and $\lim\limits_{n\rightarrow \infty}{f\over g}=0$,
respectively. 
For simplicity of notation, we set $V={n\choose k}$, $M=\prod\limits_{i=1}^{r-1}{n-ik-r+i\choose k-1}$, $N={n\choose k}-{n-r+1\choose k}$, $N_i={n-i\choose k-1}$, and 
$H={n-1\choose k-1}-{n-k-1\choose k-1}$. Note that 
$N=N_1+\cdots+N_{r-1}$, 
 $H\leq k{n-2\choose k-2}$, and 
$$(r-1){n-r+1\choose k-1}\leq N\leq (r-1){n-1\choose k-1}.$$ 
Let us remind that $r\geq 2$ is a fixed positive integer and $k\leq C n^{1\over 7}$ {\rm(}$k=o(n^{1\over 3})$ for $r=2,3${\rm)}. 
 Accordingly, we have 
$N\sim(r-1)N_{i}$ and $H=o({N_i\over k})$ for each $i\in[r-1]$. 
Moreover, $M\sim{N^{r-1}\over(r-1)^{r-1}}$ which implies 
$$p_c={\ln\left({n\choose r-1}{n-r+1\choose k}\right)\over \prod\limits_{i=1}^{r-1}{n-ik-r+i\choose k-1}}\sim  {(r-1)^{r-1}\ln\left({n\choose r-1}{n-r+1\choose k}\right)\over N^{r-1}}.$$

\subsection{\bf Proof of Theorem~\ref{maintheorem}: Part I}\label{PartI}
For the ease of reading and  without loss of generality, 
we can suppose  that $p> {\zeta\ln\left({n\choose r-1}{n-r+1\choose k}\right)\over N^{r-1}}$ and $k\leq C n^{1\over 7}$ ($k=o(n^{1\over 3})$ for $r=2,3$) for some suitable fixed 
$\zeta$ and $C$ which will be determined during the proof.  Set
$$\C=\left\{\A\subseteq {[n]\choose k}\colon  |\A|=N \mbox{ and }\A\mbox{  is not the union of any $r-1$ stars}\right\}.$$ 
Suppose that  $\A$ is an independent set of $\KG^r_{n,k}(p)$ with size $N+1$. Since there is an  $\A'\subset \A$ such that $\A'\in\C$ and $|\A'|=N$, the event that 
$\alpha(\KG^r_{n,k}(p))\geq N+1$ is a subset of the event that some member of $\C$ is an independent set of 
$\KG^r_{n,k}(p)$. 
Therefore, to prove the first part of Theorem~\ref{maintheorem}, it suffices to show that with high probability no member of $\C$ is an independent set of 
$\KG^r_{n,k}(p)$ which will be clearly done if we prove   
\begin{equation}\label{mainequation}
\sum\limits_{\A\in \C}\P\left(\A\mbox{ is an independent set of }\KG^r_{n,k}(p)\right)=o(1).
\end{equation}
 Let $\I^r_{n,k}(p)$  denote the collection of independent sets of $\KG^r_{n,k}(p)$.
For each $\A\in\C$ and $x\in[n]$, define $\A_x=\A\cap \S_x$. 
Moreover, consider fixed (with respect to $\A$) distinct elements $x_1,\ldots,x_n\in[n]$ such that 
$$|\A_{x_1}|\geq\cdots\geq |\A_{x_n}|.$$ 
Throughout the paper, we will refer to these $x_i$'s several times. For an $\A$, 
if there is more than one  choice  for $(x_1,\ldots,x_n)$, we choose one of them arbitrarily and fix it for the rest of the paper.
Now, for each $i\in[r-1]$, set $z_i=N_i- |\A_{x_i}\setminus \bigcup_{j=1}^{i-1}\A_{x_j}|$. 
Note that $\sum\limits_{i=1}^{r-1}z_i=|\A\setminus\bigcup\limits_{i=1}^{r-1}\A_{x_i}|$. 
Define 
$$
\C_1=\left\{\A\in \C\colon |\A_{x_{r-1}}|< {1\over 2r^2 k}N\right\},\qquad\qquad\C_2=\left\{\A\in \C\setminus \C_1\colon \sum\limits_{i=1}^{r-1}z_i\geq {N\over 4r^2}\right\},$$
and $$\C_3=\left\{\A\in \C\setminus \C_1\colon \sum\limits_{i=1}^{r-1}z_i<  {N\over 4r^2}\right\}.
$$
To prove Equation~(\ref{mainequation}), we will show that for each $\ell\in\{1,2,3\}$, 
\begin{equation}\label{eq2}
\sum\limits_{\A\in \C_{\ell}}\P\left(\A\in \I^r_{n,k}(p)\right)=o(1).
\end{equation} 
The rest of our discussion in this subsection is devoted to the proof of Equation~(\ref{eq2}), which will be done separately for each $\ell\in\{1,2,3\}$.

\subsection*{Proof of Equation~\ref{eq2} when $\ell=1$}
We here first need to estimate the minimum number of edges of $\KG^r_{n,k}[\A]$
when $\A\in  \C_1$. 
\begin{lemma}\label{edgelower1}
There is a constant $\eta_1=\eta_1(r)$ such that for any $\A\in  \C_1$, 
$$|E(\KG^r_{n,k}[\A])|\geq \eta_1N^r.$$ 
\end{lemma}
\begin{proof}
Let $\A\in \C_1$. According to the definition of $\C_1$, we have 
$|\A_{x_{r-1}}|< {N\over 2r^2k}$. 
Set $\A'=\A\setminus \bigcup\limits_{j=1}^{r-2}\A_{x_j}$. Note that $|\A'|\geq N_{r-1}=({1\over r-1}-o(1))N$; moreover, 
each $A\in \A'$ intersects at most $k{N\over 2r^2k}$ elements in $\A'$. 
This observation concludes in    
$$\begin{array}{lll}
|E(\KG^r_{n,k}[\A'])| & \geq & {1\over r!} \prod\limits_{i=0}^{r-1}\left(|\A'|-ik{N\over 2r^2k}\right)\\ \\
			       & \geq & {1\over r!} \left(|\A'|-{N\over 2r}\right)^{r}\\ \\ 
			       & \geq & {1\over r!} \left({1\over 2r}-o(1)\right)^rN^r\\\ \\
			       & \geq & \eta_1 N^r,
\end{array}$$
for some appropriate $\eta_1$, as desired.
  \end{proof}
By using Lemma~\ref{edgelower1}, we thus have 
$$\begin{array}{lll}
\sum\limits_{\A\in \C_1}\P\left(\A\in \I^r_{n,k}(p)\right) & \leq & |\C_1|(1-p)^{\eta_1 N^r}\\ \\
  										     & \leq & {V\choose N}e^{-p \eta_1 N^r}\\ \\
										     & \leq  & {V\choose (r-1)N_1}e^{-p \eta_1 N^r}\\ \\
										     & \leq  & \exp\left\{-p \eta_1 N^{r}+(r-1)N_1\ln{ne\over (r-1)k} \right\}\\ \\
										     & \leq  & \exp\left\{\left(-\zeta \eta_1 \ln\left({n\choose r-1}{n-r+1\choose k}\right) +(1+o(1))\ln{ne\over (r-1)k}\right)N \right\}
										    \rightarrow 0
\end{array}$$ 
provided that $\zeta > {1\over \eta_1}$, which completes the proof of Equation~\ref{eq2} for $\ell=1$. \qed
\subsection*{Proof of Equation~\ref{eq2} when $\ell=2$}
The minimum possible number of edges of $\KG^r_{n,k}[\A]$ when the size of $\A$ is given was studied by Das, Gan, and Sudakov in~\cite{Das2016}.
To state their result precisely, we first need to recall some definitions. 
We consider ${[n]\choose k}$ as a poset equipped with the lexicographical ordering: $A < B$ if $\min(A\Delta B)\in A$.
In other words, in the lexicographical ordering, we prefer sets with smaller elements. Define 
$\mathcal{L}_{n,k}(s)$ to be the set of $s$ first sets in ${[n]\choose k}$ according to the lexicographical ordering.  
\begin{alphtheorem}\cite[Theorems~1.6 and~1.7]{Das2016}\label{Dasthm}
If $n>108k^2(l+k)$ and $1\leq s\leq {n\choose k}-{n-l\choose k}$, then $\mathcal{L}_{n,k}(s)$ minimizes the number  
of edges of $\KG_{n,k}[\A]$ among all sets $\A$ of $s$ sets in ${[n]\choose k}$. 

Also, for $q\geq 3$, there is a positive constant $\eta$ such that if $n>\eta l^2k^5(l^2+k^2)e^{3q}$ and $1\leq s\leq {n\choose k}-{n-l\choose k}$, then $\mathcal{L}_{n,k}(s)$ minimizes 
the number  
of edges of $\KG^q_{n,k}[\A]$ among all sets $\A$ of $s$ sets in ${[n]\choose k}$.
\end{alphtheorem}
Although, the next corollary is a simple consequence of this theorem, for the sake of completeness, we prove it here. 
\begin{corollary}\label{DasCor}
Let $q\geq 2$ be a fixed positive integer. 
There are positive constants $\alpha$ and $\beta$ such that  for $n\geq \alpha k^7$ {\rm(}for $q=2$, $n\geq \alpha k^3${\rm)}, we have
$$|E(\KG^q_{n,k}[\A])|\geq \beta m|\A|^{q-1}$$
provided that $|\A|=N_1+\cdots+N_{q-1}+m$, where $1\leq m\leq N_q$. 
\end{corollary}
\begin{proof}
Set $s=N_1+\cdots+N_{q-1}+m$. 
In view of Theorem~\ref{Dasthm}, since $|E(\KG^q[\A])|$ will be minimized when $\A$ is the set of $s$  first sets in ${[n]\choose k}$ according to the lexicographical ordering, 
we may assume that $\A=\S_1\cup\cdots\cup\S_{q-1}\cup T$ for some $T\subseteq \S_q\setminus(\bigcup\limits_{i=1}^{q-1}\S_i)$ with $|T|=m$. 
In conclusion, one can verify that 
 $$\begin{array}{lll}
|E(\KG^q_{n,k}[\A])| & \geq & m\prod\limits_{i=1}^{q-1}{n-ik-q+i\choose k-1} \\ \\
    			      & =      & m(1-o(1)){N_1}^{q-1}\\ \\
			      & \geq & m(1-o(1))\left({1\over q}\sum\limits_{i=1}^qN_i\right)^{q-1}\\ \\
			      & \geq & \beta m|\A|^{q-1}  
\end{array}$$
for an appropriate positive constant $\beta$.
\end{proof}
Using this corollary, by the following lemma, we will prove that $\KG^{r}_{n,k}[\A]$ has many edges whenever $\A\in\C_2$. 

\begin{lemma}
There is a positive constant $\eta_2=\eta_2(r)$ such that for each $\A\in \C_2$, we have 
$$|E(\KG^{r}_{n,k}[\A])|\geq \eta_2 {N^r\over k}.$$
\end{lemma}
\begin{proof}
Consider distinct elements $x_1,\ldots,x_n\in [n]$ (as is defined fixedly above) such that 
$$|\A_{x_1}|\geq\cdots\geq |\A_{x_n}|.$$
Since  $\A\in \C_2$, we have $|\A_{x_1}|\geq \cdots\geq |\A_{x_{r-1}}|\geq  {N\over 2r^2k}$. 
Let $a\in[r-1]$ be the largest index for which $|\A_{x_a}|\geq {N\over r^2}$ (if there is no such an index, then set $a=0$).
Note that $|\A_x\cap\A_y|\leq |\S_x\cap\S_y|={n-2\choose k-2}=o({N\over k})$ for each $x\neq y\in[n]$. 
Accordingly, for each $i\leq a$, 
$$\left|\A_{x_i}\setminus \bigcup\limits_{j\in[r-1]\setminus\{i\}}\A_{x_j}\right|\geq|\A_{x_i}|- (r-2){n-2\choose k-2} \geq{N\over r^2}-o({N\over k})$$
and for each $a+1\leq i\leq r-1$,  
$$\left|\A_{x_i}\setminus \bigcup\limits_{j\in[r-1]\setminus\{i\}}\A_{x_j}\right|\geq |\A_{x_i}|- (r-2){n-2\choose k-2} \geq{N\over 2r^2k}-o({N\over k}).$$
Note that each $A\not\in \S_x$ is disjoint from all but $H$ elements in $\S_x.$ 
Consequently, if $a\geq r-2$, then 
$$\begin{array}{lll}
E(\KG^r_{n,k}[\A]) &  \geq   &\left|\A\setminus\bigcup\limits_{i=1}^{r-1}\A_{x_i}\right|\times\prod\limits_{i=1}^{r-1}\left(\left|A_{x_i}\setminus \bigcup\limits_{j\in[r-1]\setminus\{i\}}\A_{x_j}\right|-iH\right)\\ \\
			    &  \geq   & (z_1+\cdots+z_{r-1})\left({N\over r^2}-o({N\over k})\right)^{r-2}\left({N\over 2r^2k}-o({N\over k})\right)\\ \\
			    &  \geq   & {N\over 4r^2}\left({N\over r^2}-o({N\over k})\right)^{r-2}\left({N\over 2r^2k}-o({N\over k})\right)\\ \\
			    &   \geq  &  \beta'{N^{r}\over k}
\end{array}$$
for some positive constant $\beta'$(note that $H=o({N\over k})$). 
Henceforth, we assume that $a<r-2$. 
Set $\A'=\A\setminus \bigcup\limits_{i=1}^{a+1} \A_{x_i}$ and $m=N_{r-1}-{N\over r^2}$. 
Note that 
$$|\A'| \geq  N_{a+1}+\cdots+N_{r-2}+m\geq ({1\over r}-o(1))N$$ 
and 
$$m=N_{r-1}-{N\over r^2}=\left({1\over r-1}-o(1)\right)N-{N\over r^2}\geq\left({1\over r}-o(1)\right)N.$$ 
Without loss of generality, we assume that $|\A'| = N_{a+1}+\cdots+N_{r-2}+m$. 
Consequently, in view of Corollary~\ref{DasCor}, there is a constant $\beta$ for which  
$$
\begin{array}{lll}
|E(\KG^{r-a-1}_{n,k}[\A'])| & \geq & \beta\left({1\over r}-o(1)\right)N|\A'|^{r-a-2}\\ \\
	& \geq & \beta \left({1\over r}-o(1)\right)N\left(({1\over r}-o(1))N\right)^{r-a-2} \\ \\
	&  =     & \beta ({1\over r^{r-a-1}}-o(1))N^{r-a-1}.
\end{array}$$ 
Since $H=o({N\over k})$,  
$$\begin{array}{lll}
|E(\KG^{r}_{n,k}[\A])| & \geq  & |E(\KG^{r-a-1}_{n,k}[\A']|\times\prod\limits_{i=1}^{a+1}\left(\left|A_{x_i}\setminus \bigcup\limits_{j\in[a+1]\setminus\{i\}}\A_{x_j}\right|-(r-a-i)H\right)\\ \\
			        &  \geq & \beta ({1\over r^{r-a-1}}-o(1))N^{r-a-1} \left({N\over r^2}-o({N\over k})\right)^{a} \left({N\over 2r^2k}-o({N\over k})\right)\\ \\ 
			        & \geq  &\beta''{N^r\over k}
\end{array}$$
for some positive constant $\beta''$.  Setting $\eta_2=\min\{\beta',\beta''\}$  completes the proof of lemma.
\end{proof}
Now, by use of this lemma, we have 
$$\begin{array}{lll}\sum\limits_{\A\in \C_2}\P\left(\A\in \I^r_{n,k}(p)\right) &\leq& |\C_2|(1-p)^{\eta_2{N^r\over k}}\\
																		   &\leq  & {V\choose N}\exp\{-\eta_2p{N^r\over k}\}\\ \\
																		   &\leq  & {V\choose (r-1)N_1}\exp\{-\eta_2p{N^r\over k}\}\\ \\
																		   &\leq  &\exp\left\{-p \eta_2{N^r\over k}+(r-1)N_1\ln{ne\over (r-1)k} \right\}\\ \\
																		   &\leq  &\exp\left\{\left(-\zeta\eta_2 {1\over k}\ln\left({n\choose r-1}{n-r+1\choose k}\right) +\left(1+o(1)\right)\ln{ne\over (r-1)k}\right)N \right\}\\ \\
																		   &\leq  &\exp\left\{\left(-\zeta\eta_2 \ln\left({n-r+1\over k}\right) +\left(1+o(1)\right)\ln{ne\over (r-1)k}\right)N \right\}\rightarrow 0
\end{array}$$ 
provided that $\zeta>{1\over \eta_2}$. \qed

\subsection*{Proof of Equation~\ref{eq2} when $\ell=3$}
For each $\A\in\C_3$ and each $i\in[r-1]$, we clearly have  $$\left|\A_{x_i}\setminus \bigcup\limits_{j=1}^{i-1}\A_{x_j}\right|= N_i-z_i
\geq\left({1\over r-1}-o(1)\right)N- z_i.$$ 
Consequently,  $$
\begin{array}{lll}
\left|\A_{x_i}\setminus \bigcup\limits_{j\in[r-1]\setminus\{i\}}\A_{x_j}\right| &  \geq & \left|\A_{x_i}\setminus \bigcup\limits_{j=1}^{i-1}\A_{x_j}\right|-\sum\limits_{j=i+1}^{r-1}|\A_{x_i}\cap\A_{x_j}|\\ \\ 
											     & \geq & N_i-z_i -(r-2){n-2\choose k-2}\\ \\ 
											     & \geq  &\left({1\over r-1}-o(1)\right)N- z_i.
\end{array}$$
Accordingly, since $H=o(N)$, for large enough $n$, we have 
$$\begin{array}{lll}
E(\KG^r_{n,k}[\A]) &  \geq & \left|\A\setminus \bigcup\limits_{i=1}^{r-1}\A_{x_i}\right|\times\prod\limits_{i=1}^{r-1}\left(\left|\A_{x_i}\setminus \bigcup\limits_{j\in[r-1]\setminus\{i\}}\A_{x_j}\right|-iH\right)\\ \\
			    &     =   & (z_1+\cdots+z_{r-1})\prod\limits_{i=1}^{r-1}\left(\left|\A_{x_i}\setminus \bigcup\limits_{j\in[r-1]\setminus\{i\}}\A_{x_j}\right|-iH\right)\\ \\
			    &      \geq  &  (z_1+\cdots+z_{r-1})\prod\limits_{i=1}^{r-1}({N\over r}-z_i).
\end{array}$$
Hence, if we define $f(z_1,\ldots,z_{r-1})=(z_1+\cdots+z_{r-1})\prod\limits_{i=1}^{r-1}({N\over r}-z_i)$, then, for large enough $n$,  
$$\begin{array}{lll}\sum\limits_{\A\in \C_3}\P\left(\A\in \I^r_{n,k}(p)\right)& \leq & \sum\limits_{1\leq z_1+\cdots+z_{r-1}\leq c N}{n\choose r-1}{N_1\choose z_1}\cdots{N_{r-1}\choose z_{r-1}}{V\choose z_1+\cdots+z_{r-1}}    (1-p)^{f(z_1,\ldots,z_{r-1})}\\ \\
& \leq & \sum\limits_{1\leq z_1+\cdots+z_{r-1}\leq c N}{n\choose r-1}{N_1\choose z_1}\cdots{N_{r-1}\choose z_{r-1}}{V\choose z_1+\cdots+z_{r-1}}    e^{-p f(z_1,\ldots,z_{r-1})}.
\end{array}$$
Now, we set $$g(z_1,\ldots,z_{r-1})={n\choose r-1}{N_1\choose z_1}\cdots{N_{r-1}\choose z_{r-1}}{V\choose z_1+\cdots+z_{r-1}}    e^{-p f(z_1,\ldots,z_{r-1})}.$$ 
It is simple to check that there is a constant $\zeta_0$ such that for $\zeta>\zeta_0$, 
if $\sum\limits_{i=1}^{r-1}z_i\geq 2$, then  for each $z_i\geq 1$, 
$$\begin{array}{lll}
{g(z_1,\ldots,z_i,\ldots,z_{r-1})\over g(z_1,\ldots,z_i-1,\ldots,z_{r-1})} &= &{{N\choose z_i}{V\choose z_1+\cdots+z_{r-1}}\over {N\choose z_i-1}{V\choose z_1+\cdots+z_{r-1}-1}}e^{-p \big(f(z_1,\ldots,z_{r-1})-f(z_1,\ldots,z_i-1,\ldots,z_{r-1})\big)}=o(1).
\end{array}$$
Therefore, for sufficiently large $n$,  we have 
$${g(z_1,\ldots,z_i,\ldots,z_{r-1})\over g(z_1,\ldots,z_i-1,\ldots,z_{r-1})}<1$$ 
which clearly concludes in $$g(z_1,\ldots,z_{r-1})\leq g(1,0,\ldots,0).$$ 
This implies that there is a constant $c=c(r)$ for which 
$$\begin{array}{lll}
\sum\limits_{\A\in \C_3}\P\left(\A\in \I^r_{n,k}(p))\right) & \leq & \sum\limits_{1\leq z_1+\cdots+z_{r-1}\leq {N\over 4r^2}} g(1,0,\ldots,0)\\ \\
													  & \leq & \sum\limits_{1\leq z_1+\cdots+z_{r-1}\leq {N\over 4r^2}}  {n\choose{r-1}}N_1Ve^{-p cN^{r-1}}\\ \\
													  & \leq & {n\choose r-1}N^{r}Ve^{-p cN^{r-1}}\\ \\
													  &   =   & \exp\left\{-p cN^{r-1}+\ln({n\choose{r-1}}N^rV)\right\}\\ \\
													  & \leq & \exp\left\{-\zeta c\ln\left({n\choose r-1}{n-r+1\choose k}\right) +\ln({n\choose r-1}N^rV)\right\}\\ \\
													  &   =   & {{n\choose r-1}N^rV \over \left({n\choose r-1}{n-r+1\choose k}\right)^{ c\zeta}}\rightarrow 0
\end{array}$$
provided that $\zeta >{r+1\over c}$. \qed

\bigskip
We are now ready to finish the proof of Theorem~\ref{maintheorem}: Part I. 
\begin{proof}[Completing the proof of Theorem~\ref{maintheorem}: Part I] 
In conclusion, if we set $\zeta >\max\{\zeta_0, {1\over \eta_1},{1\over \eta_2},{r+1\over c}\}$, then for all $\ell\in\{1,2,3\}$, we simultaneously have 
$$
\sum\limits_{\A\in \C_{\ell}}\P\left(\A\in \I^r_{n,k}(p)\right)=o(1)
$$
finishing the proof. 
\end{proof}

\subsection{\bf Proof of Theorem~\ref{maintheorem}: Part II}\label{PartII}
It should be noticed that our proof is similar to that of the second part of Theorem~\ref{bela1} in~\cite{MR3403515}.
Let $ p\leq(1-\varepsilon)p_c$ for some constant $\varepsilon\in(0,1]$. Here we prove that 
$$\mathbb{P}\left(\alpha\left(\KG^r_{n,k}(p)\right)\leq{n\choose k}-{n-r+1\choose k}\right)=o(1).$$
Let $Y$ denote the random variable counting the number of  pairs $(A,Q)$ such that $Q\in{[n]\choose r-1}$, $A\not\in \S_Q=\bigcup\limits_{x\in Q}\S_x$, and $E(\KG^r_{n,k}(p)[\S_Q\cup\{A\}])=\varnothing$. Clearly, to prove the desired assertion, it suffices to show that $\P(Y>0)$ goes to $1$ as $n$ tends to infinity. 
Let us remind that $M=\prod\limits_{i=1}^{r-1}{n-ik-1\choose k-1}$.
It is easy to check  that 
$$\begin{array}{lll}
\mathbb{E}[Y]&=&{n \choose r-1}{n-r+1 \choose k}(1-p)^M\\ \\
 &\geq & {n \choose r-1}{n-r+1 \choose k}\exp(-(p+p^2)M)\\ \\
 &\geq & {n \choose r-1}{n-r+1 \choose k}\exp\left\{-(1+p)(1-\varepsilon)\ln\left({n \choose r-1}{n-r+1 \choose k}\right)\right\}\\ \\
 &\geq & ({n \choose r-1}{n-r+1 \choose k})^{\varepsilon-p+\varepsilon p}.
  \end{array}
$$
Therefore, $\mathbb{E}[Y]\to \infty$ when $p\leq (1-\varepsilon)p_c.$
Hence, by using the classical second moment technique, to prove that $\P(Y>0)\rightarrow 1$, it is  suffices to show that $\mathrm{Var}[Y]=o(\mathbb{E}[Y]^2).$ 
Let $Y'$ denote the random variable counting the number of  $4$-tuples $(A,B,Q,T)$ with 
$Q, T\in {[n]\choose r-1}$, $A\in{[n]\choose k} \setminus\S_{Q}$  and $B\in{[n]\choose k} \setminus\S_{Q}$ such that $(A,Q)\neq (B,T)$
and $$E\left(\KG^r_{n,k}(p)[\S_Q\cup\{A\}]\right)=E\left(\KG^r_{n,k}(p)[\S_T\cup\{B\}]\right)=\varnothing.$$
Clearly, 
$$\mathbb{E}[Y']=\sum\mathbb{P}\left(\S_Q\cup\{A\},\S_T\cup\{B\} \in\I^r_{n,k}\right),$$
where the summation is taken over all  ordered $4$-tuples $(A,B,Q,T)$ with  
$Q, T\in {[n]\choose r-1}$, $A\in{[n]\choose k} \setminus\S_{Q}$, $B\in{[n]\choose k} \setminus\S_{Q}$, and $(A,Q)\neq (B,T)$. 
Now, one can verified that 
$$\begin{array}{lll}
\sum\limits_{Q\neq T}\mathbb{P}\left(\S_Q\cup\{A\},\S_T\cup\{B\} \in\I^r_{n,k}\right)&\leq& {n \choose r-1}^2{n-r+1 \choose k}^{2} (1-p)^{2M-{ O(N^{r-2})}}\\ \\
&=&(1+o(1))\mathbb{E}[Y]^2
  \end{array}$$
and 
$$\begin{array}{lll}
\sum\limits_{Q= T, A\neq B}\mathbb{P}\left(\S_Q\cup\{A\},\S_T\cup\{B\} \in\I^r_{n,k}\right)&\leq& {n \choose r-1}{n-r+1 \choose k}^2 (1-p)^{2M}\\ \\
&=&o(\mathbb{E}[Y]^2).
  \end{array}$$
Note that $$\mathrm{Var}[Y]=\mathbb{E}[Y^2]-\mathbb{E}[Y]^2=\mathbb{E}[Y]+\mathbb{E}[Y']-\mathbb{E}[Y]^2.$$ 
Hence, $\mathrm{Var}[Y]=\mathbb{E}[Y]+o(\mathbb{E}[Y^2])=o(\mathbb{E}[Y^2])$, as desired. \qed
\section*{Acknowledgments}
The research of Meysam Alishahi was in part supported by a grant from IPM (No. 96050013).


\begin{thebibliography}{10}

\bibitem{AH2018}
M.~Alishahi and H.~Hajiabolhassan.
\newblock Chromatic number of random {K}neser hypergraphs.
\newblock {\em J. Combin. Theory Ser. A}, 154:1 -- 20, 2018.

\bibitem{MR3482268}
J.~Balogh, B.~Bollob{\'a}s, and B.~P. Narayanan.
\newblock Transference for the {E}rd{\H o}s--{K}o--{R}ado theorem.
\newblock {\em Forum Math. Sigma}, 3:e23, 18, 2015.

\bibitem{BGPR14}
L.I. Bogolyubskii, A.S. Gusev, M.M. Pyaderkin, and A.M. Raigorodskii.
\newblock The independence numbers and chromatic numbers of random subgraphs in
  some sequences of graphs.
\newblock {\em Dokl. Akad. Nauk}, 457(4):383--387, 2014.

\bibitem{BGPR15}
L.I. Bogolyubskii, A.S. Gusev, M.M. Pyaderkin, and A.M. Raigorodskii.
\newblock The independence numbers and the chromatic numbers of random
  subgraphs of some distance graphs.
\newblock {\em Mat. Sb.}, 206(10):3--36, 2015.

\bibitem{BolDayErd76}
B.~Bollob\'as, D.E. Daykin, and P.~Erd{\H o}s.
\newblock Sets of independent edges of a hypergraph.
\newblock {\em Quart. J. Math. Oxford Ser. (2)}, 27(105):25--32, 1976.

\bibitem{MR3403515}
B.~Bollob{\'a}s, B.~P. Narayanan, and A.M. Raigorodskii.
\newblock On the stability of the {E}rd{\H o}s-{K}o-{R}ado theorem.
\newblock {\em J. Combin. Theory Ser. A}, 137:64--78, 2016.

\bibitem{Das2016}
S.~Das, W.~Gan, and B.~Sudakov.
\newblock The minimum number of disjoint pairs in set systems and related
  problems.
\newblock {\em Combinatorica}, 36(6):623--660, Dec 2016.

\bibitem{MR3504983}
S.~Das and T.~Tran.
\newblock Removal and stability for {E}rd{\H o}s-{K}o-{R}ado.
\newblock {\em SIAM J. Discrete Math.}, 30(2):1102--1114, 2016.

\bibitem{MR3513857}
P.~Devlin and J.~Kahn.
\newblock {On ``stability'' in the Erd\H os-Ko-Rado theorem.}
\newblock {\em {SIAM J. Discrete Math.}}, 30(2):1283--1289, 2016.

\bibitem{Erdos65}
P.~Erd{\H o}s.
\newblock A problem on independent {$r$}-tuples.
\newblock {\em Ann. Univ. Sci. Budapest. E\"otv\"os Sect. Math.}, 8:93--95,
  1965.

\bibitem{ErdGal59}
P.~Erd{\H o}s and T.~Gallai.
\newblock On maximal paths and circuits of graphs.
\newblock {\em Acta Math. Acad. Sci. Hungar}, 10:337--356 (unbound insert),
  1959.

\bibitem{Frankl13}
P.~Frankl.
\newblock Improved bounds for {E}rd{\H o}s' matching conjecture.
\newblock {\em J. Combin. Theory Ser. A}, 120(5):1068--1072, 2013.

\bibitem{Frankl2017}
P.~Frankl.
\newblock On the maximum number of edges in a hypergraph with given matching
  number.
\newblock {\em Discrete Appl. Math.}, 216(part 3):562--581, 2017.

\bibitem{F18}
P.~Frankl, J.~Han, H.~Huang, and Y.~Zhao.
\newblock A degree version of the {H}ilton-{M}ilner theorem.
\newblock {\em J. Combin. Theory Ser. A}, 155:493--502, 2018.

\bibitem{FK16}
P.~Frankl and A.~Kupavskii.
\newblock Two problems of {P}. {E}rd{\H o}s on matchings in set families.
\newblock {\em arXiv preprint arXiv:1607.06126}, 2016.

\bibitem{FK17}
P.~Frankl and A.~Kupavskii.
\newblock The largest families of sets with no matching of sizes 3 and 4.
\newblock {\em arXiv preprint arXiv:1701.04107}, 2017.

\bibitem{FranLucMie12}
P.~Frankl, T.~{\L}uczak, and K.~Mieczkowska.
\newblock On matchings in hypergraphs.
\newblock {\em Electron. J. Combin.}, 19(2):Paper 42, 5, 2012.

\bibitem{HM}
A.J.W. Hilton and E.C. Milner.
\newblock Some intersection theorems for systems of finite sets.
\newblock {\em Quart. J. Math. Oxford Ser. (2)}, 18:369--384, 1967.

\bibitem{HuaLohSud12}
H.~Huang, P.-S. Loh, and B.~Sudakov.
\newblock The size of a hypergraph and its matching number.
\newblock {\em Combin. Probab. Comput.}, 21(3):442--450, 2012.

\bibitem{KR17}
S.G. Kiselev and A.M. Raigorodskii.
\newblock On the chromatic number of a random subgraph of the kneser graph.
\newblock {\em Doklady Mathematics}, 96(2):475--476, Sep 2017.

\bibitem{K16}
A.~Kupavskii.
\newblock On random subgraphs of {K}neser and {S}chrijver graphs.
\newblock {\em J. Combin. Theory Ser. A}, 141:8--15, 2016.

\bibitem{K17}
A.~{Kupavskii}.
\newblock {Random Kneser graphs and hypergraphs}.
\newblock {\em ArXiv e-prints}, December 2016.

\bibitem{KZ18}
A.~Kupavskii and D.~Zakharov.
\newblock Regular bipartite graphs and intersecting families.
\newblock {\em J. Combin. Theory Ser. A}, 155:180--189, 2018.

\bibitem{LucMie14}
T.~{\L}uczak and K.~Mieczkowska.
\newblock On {E}rd{\H o}s' extremal problem on matchings in hypergraphs.
\newblock {\em J. Combin. Theory Ser. A}, 124:178--194, 2014.

\bibitem{P15}
M.M. Pyaderkin.
\newblock On the stability of the {E}rd{\H o}s-{K}o-{R}ado theorem.
\newblock {\em Doklady Mathematics}, 91(3):290--293, May 2015.

\bibitem{P17}
M.M. Pyaderkin.
\newblock On the stability of some {E}rd{\H o}s-{K}o-{R}ado type results.
\newblock {\em Discrete Math.}, 340(4):822--831, 2017.

\bibitem{PR16}
M.M. Pyaderkin and A.M. Raigorodskii.
\newblock On random subgraphs of a {K}neser graph and its generalizations.
\newblock {\em Dokl. Akad. Nauk}, 470(4):384--386, 2016.

\bibitem{R17}
A.M. Raigorodskii.
\newblock On the stability of the independence number of a random subgraph.
\newblock {\em Doklady Mathematics}, 96(3):628--630, Nov 2017.

\end{thebibliography}
\end{document}